\documentclass[12pt]{amsart}
\usepackage{pdfsync}
\usepackage{graphicx,color}
\usepackage{amssymb}
\usepackage{amsmath}
\usepackage{amsthm}
\usepackage{enumerate}
\usepackage{subcaption}
\usepackage{esvect}
\usepackage[left=3cm,top=3.8cm,right=3cm]{geometry} 
\usepackage{ucs}
\usepackage{cite}
\usepackage{amsxtra}
\usepackage{epstopdf}
\usepackage{bbold}
\usepackage[utf8x]{inputenc}
\usepackage{verbatim}                                  
\usepackage[toc,page]{appendix}                        
\usepackage{pdfsync} 
\usepackage{esint}
\usepackage{bookmark}
\usepackage{hyperref}
\hypersetup{pdfstartview={FitH}}

\allowdisplaybreaks

\newcommand{\e}{\varepsilon}

\DeclareMathOperator{\diam}{diam}
\DeclareMathOperator{\dist}{dist}
\DeclareMathOperator{\card}{card}

\newcommand{\be}{\begin{equation}}
\newcommand{\ee}{\end{equation}}

\newtheorem{theorem}{Theorem}[section]
\newtheorem{definition}[theorem]{Definition}
\newtheorem{proposition}[theorem]{Proposition}

\newtheorem{lemma}[theorem]{Lemma}
\newtheorem{example}[theorem]{Example}
\newtheorem{question}[theorem]{Question}

\newcommand{\R}{\mathbb{R}}

\def\N{{\mathbb N}}

\makeatletter
\def\moverlay{\mathpalette\mov@rlay}
\def\mov@rlay#1#2{\leavevmode\vtop{%
   \baselineskip\z@skip \lineskiplimit-\maxdimen
   \ialign{\hfil$\m@th#1##$\hfil\cr#2\crcr}}}
\newcommand{\charfusion}[3][\mathord]{
    #1{\ifx#1\mathop\vphantom{#2}\fi
        \mathpalette\mov@rlay{#2\cr#3}
      }
    \ifx#1\mathop\expandafter\displaylimits\fi}
\makeatother

\numberwithin{equation}{section}

\begin{document}
\thispagestyle{empty}

\title{The density of sets containing large similar copies of finite sets}

\author{Kenneth Falconer}
\address{Kenneth Falconer,
School of Mathematics and Statistics,
University of St.~Andrews}
\email{kjf@st-andrews.ac.uk}

\author{Vjekoslav Kova\v{c}}
\address{Vjekoslav Kova\v{c},
Department of Mathematics, Faculty of Science,
University of Zagreb, Bijeni\v{c}ka cesta 30, 10000 Zagreb, Croatia}
\email{vjekovac@math.hr}
\thanks{VK is supported by the Croatian Science Foundation, project n$^{\circ}$ UIP-2017-05-4129 (MUNHANAP)}

\author{Alexia Yavicoli}
\address{Alexia Yavicoli,
School of Mathematics and Statistics, University of St.~Andrews, UK.
Current address: Department of Mathematics, the University of British Co\-lum\-bia. 1984 Mathematics Road, Vancouver BC V6T 1Z2, Canada}
\email{yavicoli@math.ubc.ca, alexia.yavicoli@gmail.com}
\thanks{AY is supported by the Swiss National Science Foundation, grant n$^{\circ}$ P2SKP2\_184047}

\keywords{pattern, density, similarity, arithmetic progression, discrepancy}
\subjclass[]{28A12, 05D10 \and 11B25}

\begin{abstract}
We prove that if $E \subseteq \R^d$ ($d\geq 2$) is a Lebesgue-measurable set with density larger than $(n-2)/(n-1)$, then $E$ contains  similar copies of every $n$-point set $P$ at all sufficiently large scales. Moreover, `sufficiently large' can be taken to be uniform over all $P$ with prescribed size, minimum separation and diameter. On the other hand, we construct an example to show that the density required to guarantee all large similar copies of $n$-point sets tends to $1$ at a rate $1- O(n^{-1/5}\log n)$.
\end{abstract}


\maketitle



\section{Introduction}

In this paper a finite subset $P$ of $\R^d$ with at least two distinct points will be called a {\em pattern}. There are many ways of viewing the question of finding necessary or sufficient conditions on a set  $E \subseteq\R^d$  to contain some, or many, similar (or alternatively homothetic or congruent) copies of a given pattern $P$. Here we will be concerned with finding conditions that guarantee that $E$ contains scaled similar copies of $P$ for all sufficiently large scalings. We will assume throughout that $d\geq 2$ and that $E$ is $\mathcal{L}^d$-measurable, where $\mathcal{L}^d$ denotes the Lebesgue measure on $\R^d$. The Euclidean norm of $x\in\R^d$ will simply be written as $\|x\|$. We denote by $B_r(x)\subseteq \R^d$ the closed ball of centre $x$ and radius $r$; we will abbreviate this to $B_r$ for any ball of radius $r$ when the centre is not relevant. The {\em upper Banach density} of $E \subseteq \R^d$ is defined by
\be\label{upbden}
\rho:=\rho(E):=\limsup_{r \to +\infty} \sup_{x \in \R^d}\frac{\mathcal{L}^d(E \cap B_r(x))}{\mathcal{L}^d(B_r)}
\ee
and the usual {\em upper density} by
\be\label{upden}
\overline{d}(E):=\limsup_{r \to +\infty}\frac{\mathcal{L}^d(E \cap B_r(0))}{\mathcal{L}^d(B_r)},
\ee
{these definitions being independent of the origin chosen.} It is clear that $\rho(E) \geq \overline{d}(E)$ and the inequality can be strict.
{Consequently, results that assume a lower bound on $\rho(E)$ are more general than the corresponding results that assume a lower bound on $\overline{d}(E)$.}

The most basic result of this kind is for 2-point patterns: for every Lebesgue-measurable set $E\subseteq \R^2$ with {positive upper Banach density} there exists $R>0$ such that all distances greater than $R$ are realised between the points of $E$. This problem was posed by Sz\'{e}kely \cite{S82} and several proofs  were given in the 1980s, by Falconer and Marstrand \cite{FM86} with a geometric proof, by Bourgain \cite{B86} for $\R^d$ with $d\geq 2$ using harmonic analysis and by Furstenberg, Katznelson and Weiss \cite{FKW90} using ergodic theory. More recently Quas \cite{Q09} gave a more combinatorial proof.

Rice \cite{R20} showed that the positive density requirement cannot be weakened. For all $d\geq 1$ and any function $f\colon (0, \infty) \to [0, 1]$ with $\lim_{r\to\infty}f(r)= 0$ he constructed a measurable set $E \subseteq \R^d$ and a sequence $r_n \to \infty$ such that $\|x-y\|\neq r_n$ for all $x, y \in E$,  with $\mathcal{L}^d(E \cap B_{r_n}(0))/\mathcal{L}^d(B_{r_n}(0)) \geq f(r_n)$ for all $n \in \N$, where $B_r(x)$ is the ball of radius $r$ centred at $x$.

It is natural to consider analogous questions for patterns with more than two points. Indeed, Bourgain's paper also showed that a set of positive upper density $E \subseteq \R^d$ contains all sufficiently large similar copies of every $d$-point pattern provided that the points span a $(d-1)$-dimensional hyperplane, see \cite{HLM17,Kov20,LM18,LM20} for various other proofs. On the other hand, he showed by the following example, which relies on the parallelogram identity, that this spanning condition is necessary.

\begin{example}\label{ex:Bourgain} {\rm\cite{B86}}.
Let  $0< s<\frac{1}{4}$ and let $E  := \{x\in \R^d : \|x\|^2 \in [0,s] + (\mathbb{N}\cup\{0\})\}$, so that $E$ is a union of annuli and has density $s$. Then there are arbitrarily large values of $r$ such that $E$ contains no congruent copy of $\{0,r,2r\}$.
\end{example}

Subsequently, Graham showed that a similar conclusion holds for any non-spherical set.

\begin{example}{\rm\cite{Gra94}}.
Let $P\subseteq \mathbb{R}^d$ be a finite set of points that do not all lie on the surface of any $(d-1)$-sphere. Then there is a set $E$ of positive upper density and arbitrarily large values of $r$ such that $E$ does not contain a congruent copy of $rP$ .
\end{example}

It is an open question whether every plane set of positive upper density contains all large copies of every non-degenerate triangle. However,  Furstenberg, Katznelson and Weiss \cite[Theorem B]{FKW90} showed that if $E\subseteq \R^2$ has positive upper density, then every $\delta$-neighbourhood of $E$  contains all sufficiently large similar copies of every triangle, and Ziegler \cite{Z06} extended this to larger patterns (i.e., simplices, possibly degenerate) in $\R^d$ for $d\geq 2$.

As far as other configurations go,
Morris \cite{M15} showed that in any set of positive density one can find triangles with all sufficiently large (compatible) perimeters and areas. Lyall and Magyar \cite{LM18} considered products of $k$- and $k'$-simplices in $\R^d$ where $ k+k'+6\leq d$ and in particular showed that, given the vertices of a rectangle $P$, any subset of $\R^d \ (d\geq 4)$ with positive upper Banach density contains all sufficiently large similar copies of $P$.
Generalizations to multiple products were discussed in \cite{DK18} and \cite{LM19}; the latter paper successfully handled arbitrary products of non-degenerate simplices.
Moreover, Lyall and Magyar \cite{LM20} showed that sets of positive upper Banach density in $\R^d$ contain all large enough copies of `proper $k$-degenerate distance graphs' if $d\geq k+1$; for example when $k=1$ these include finite trees and chains with prescribed edge lengths. Here the position of the vertices of the graphs in the large copies is immaterial provided that the scaled distances between vertices are realised.
It is also possible to study analogous questions for anisotropic patterns, that is, for families of point configurations with power-type dependence on a real parameter which might be thought of as their size, see \cite{Kov20}.
Finally, several authors have got around Example \ref{ex:Bourgain} by investigating these questions when $\R^d$ is endowed with the $\ell^p$-norm for $1\leq  p\leq \infty$, $p\neq 2$, see \cite{CMP17,DK20,DK18,DVL18,K03}.

\smallskip
Given such conclusions  it is natural to seek general sufficient conditions that ensure that a set contains all sufficiently large similar copies of a given pattern. For homothetic copies (i.e., when we do not allow rotations) {an easy} argument establishes the following statement {and we do not even need to confine ourselves to large scales.}

\begin{proposition}\label{easypropintr}
Let $E \subseteq \R^d$ have upper {Banach} density $\rho> \frac{n-1}{n}$ and let $P$ be an $n$-point pattern in $\R^d$.
{Then for every $r>0$ the set $E$ contains a translated copy of $rP$.}
\end{proposition}

{The proof of Proposition~\ref{easypropintr} is straightforward. By the definition of $\rho(E)$ we can find $x\in\R^d$ and $r'>0$ such that $E \cap B_{r'}(x)$ occupies a proportion strictly greater than $\frac{n-1}{n}$ of the ball centred at $x$ with radius $r'+r\max_{y\in P}\|y\|$. Translates of $E \cap B_{r'}(x)$ by the vectors from $-rP$ remain inside this enlarged ball, so they all share a common point. We omit the details.}

An aim of this paper is to obtain a quantitatively stronger result, that sets of density greater than  $\frac{n-2}{n-1}$ contain all sufficiently large similar copies of $n$-point patterns in a sense that is uniform over certain patterns of a fixed size. For a pattern $P=\{x_0, \ldots, x_{n-1}\} \subseteq \R^d$ we write $\mathop{{\rm sep}} P  = \min_{i\neq j}\|x_i - x_j\|$ for the {\it minimum separation} of $P$ and ${\diam P} = \max_{i\neq j}\|x_i - x_j\|$ for the {\it diameter} of $P$. By allowing rotations, the density of $E$ that guarantees similar copies of $P$ does not have to be as large as for homothetic copies.

\begin{theorem}\label{mainth}
Let $E \subseteq \R^d$ have upper Banach density $\rho >\frac{n-2}{n-1}$.
Then there exists a number $R:=R(E,S,D,n)>0$ such that, for every $n$-point pattern $P\subseteq \R^d$ satisfying $S\leq \mathop{{\rm sep}} P \leq  {\diam P}\leq D$, if $r\geq R$, then there exist $z_r \in \R^d$ and a rotation $Q_r\in SO(d)$ such that $r Q_r (P)+z_r \subseteq E$, i.e., $E$ contains a similar copy of $P$ at all scales at least $R$.
\end{theorem}

To prove Theorem \ref{mainth} we develop a quantitative version of the argument by Falconer and Marstrand \cite{FM86}, which we extend to $\R^d$ for $d\geq 2$. Note that, because in the proofs of Proposition \ref{concenspheres} and Lemma \ref{rotation} we choose $x$ and $Q$ to be any points in certain sets of positive ${\mathcal L}^d$-measure and $\sigma$-measure respectively,  there will be a set of isometries of positive $(\sigma \times \mathcal{L}^d)$-measure under which copies of $P$ at a (large) given scale will be contained in $E$.

It is natural to ask for the minimum upper Banach density required  in Theorem \ref{mainth}: to what extent can the value $\frac{n-2}{n-1} = 1 -\frac{1}{n-1}$ be reduced, and how does the density required to guarantee the presence of all sufficiently large copies of all $n$-point patterns behave as $n\to\infty$?  Using arithmetic sequences we show that this density must approach 1 as $n$ gets large, indeed at a rate  $1- O(n^{-1/5}\log n)$.
The logarithm function is understood to have the number $e$ as its base.

\begin{theorem}\label{thm:lowerbound}
For all $n\in \mathbb{N}$ $(n\geq 2)$ and $d\in \mathbb{N}$ there exists a measurable set $E=E(d,n) \subseteq \mathbb{R}^d$ of density at least
\be\label{eq:lowerbound}
1-\frac{10 \log n}{n^{1/5}}
\ee
such that there are arbitrarily large values of $r$ for which $E$ contains no congruent copy of $\{0,r,2r,\ldots,(n-1)r\}$.
\end{theorem}

{Somewhat surprisingly, Theorem~\ref{thm:lowerbound} will be reduced to density and equidistribution properties of quadratic sequences on the torus $\mathbb{T}=\mathbb{R}/\mathbb{Z}$.
Considering other patterns than just arithmetic progressions could lead to interesting problems on Diophantine approximations.}

Theorems \ref{mainth} and \ref{thm:lowerbound} leave open the following question.

\begin{question}
What is the smallest $0\leq\rho_{\textup{min}}(d,n)<1$ such that every measurable set $E\in \mathbb{R}^d$ of upper {Banach} density $\rho>\rho_{\textup{min}}(d,n)$ contains all sufficiently  large scale similar copies of all $n$-point patterns?
Theorems \ref{mainth} and \ref{thm:lowerbound} give
\[ 1-\frac{10 \log n}{n^{1/5}} \leq \rho_{\textup{min}}(d,n) \leq 1 -\frac{1}{n-1}. \]
Is it possible to improve either one of the two asymptotic bounds $1- O(n^{-1/5}\log n)$ and $1- O(n^{-1})$ as $n\to\infty$?
\end{question}

\smallskip
We remark in passing that problems of a similar nature are widely studied in the context of null Lebesgue measure, where one seeks conditions on the Hausdorff dimension or thickness of a set to guarantee that it contains a similar copy of a pattern.
In particular, {\L}aba and Pramanik \cite{LP09} gave conditions on fractal sets in the real line that ensure the existence of an arithmetic progression of length $3$. Then Henriot, {\L}aba and Pramanik \cite{HLP16} and Chan, {\L}aba and Pramanik  \cite{CLP16} improved the hypotheses and obtained results for more general patterns in $\R^d$. Iosevich and Liu \cite{IL19}  made a further improvement in $\R^4$ for copies of triangles. See also \cite{IT19,GILP16,GIM18}, where patterns are guaranteed in sets of sufficiently large Hausdorff dimension, and \cite{Y20} for sets of large enough thickness.

\smallskip
The proof of Theorem \ref{mainth} spans Sections \ref{sec:finitepatterns} and \ref{sec:keyestimates}, while the proof of Theorem \ref{thm:lowerbound} will be given in Section \ref{sec:lowerbound}.

\section{Proof of Theorem \ref{mainth}: upper bound for $\rho_{\textup{min}}$}
\label{sec:finitepatterns}

In this section we will prove Theorem \ref{mainth}, that is, we will show that a set $E$ contains similar copies of a pattern $P$ at all sufficiently large scalings provided that the {upper} Banach density of $E$ is {greater than $\frac{n-2}{n-1}$}. We will also see that `sufficiently large' can be taken to be uniform over the patterns $P$ satisfying $\card P = n$ and $S\leq \mathop{{\rm sep}} P \leq  {\diam P}\leq D$.

{We require the fact, stated in Proposition \ref{concenspheres},} that if $E$ has upper Banach density $\rho$ and $\rho'<\rho$  then any given family of a finite number of concentric spheres can be scaled and translated so that a proportion at least $\rho'$ of each spherical surface is in $E$, for all sufficiently large scalings.
{We denote by $S_r(x)$ the sphere of centre $x$ and radius $r$; we will abbreviate this to $S_r$ for any sphere of radius $r$ when the centre is not relevant.}

\begin{proposition}\label{concenspheres}
Let $E \subseteq \R^d$ be measurable and let $0<\rho' <\rho (E)$ and $0<S\leq D$.
Then there is an $s_0:=s_0(E,S,D,m)>0$ such that, for every set of numbers $\{r_i\}_{i=1}^{m}$ with $ r_i \in [S,D]$ for all $i$, for all $s\geq s_0$ there exists $x\in E$ such that
\be\label{bigint}
\mathcal{L}^{d-1} (E\cap S_{r_i s}(x) ) > \rho' \mathcal{L}^{d-1} (S_{r_i s})
\ee
for all $1\leq i\leq m$.
\end{proposition}
The technical proof of this {proposition} is given in the next section.

Next, we show that, given a pattern $P:=\{x_0, \ldots,x_{n-1}\}\subseteq \R^d$, if the proportion of the set $E$ in every sphere centred in $x_0$ with radius $\|x_i-x_0\|$ is larger than $\frac{n-2}{n-1}$, then there is a rotated copy of $P$ about $x_0$ that is contained in the set $E$.
For {this pattern $P$} we write $r_i:=\| x_i -x_0\|>0$ for  $1 \leq i \leq n-1$.
Let $SO(d)$ be the special orthogonal group of rotations of $\R^d$ about the origin and let $\sigma$ be normalised Haar measure on $SO(d)$.

\begin{lemma}\label{rotation}
Let $E\subseteq \R^d$ be measurable and let $P:=\{x_0, \ldots,x_{n-1}\} \subseteq \R^d$ be a pattern. Suppose that $x_0 \in E$ and
\be\label{mescond}
\mathcal{L}^{d-1} (E\cap S_{r_i}(x_0) ) > \Big(\frac{n-2}{n-1}\Big) \mathcal{L}^{d-1} (S_{r_i}) \quad (1\leq i \leq n\! -\! 1).
\ee
Then there exists $Q\in SO(d)$ such that $Q(P-x_0)+x_0 \subseteq E$, i.e., $P$ may be rotated about $x_0$ so that $x_i\in E$ for all $0\leq i\leq n-1$.
\end{lemma}

\begin{proof}
Without loss of generality take $x_0 = 0$ so that $x_i\in S_{r_i}(0)$ for $1\leq i\leq n-1$. From \eqref{mescond},
$$\sigma \big\{Q\in SO(d): Q(x_i) \in E\cap S_{r_i}(0)\big\} = \frac{\mathcal{L}^{d-1} (E\cap S_{r_i}(0))}{\mathcal{L}^{d-1} (S_{r_i}(0) )} >\frac{n-2}{n-1}.$$
Hence
\begin{align*}
\sigma \big\{Q\in SO(d): Q(x_i) &\in E\cap S_{r_i}(0) \mbox{ for all } 1\leq i \leq n\! -\! 1\big\}\\
&\geq \sigma (SO(d)) -\sum_{i=1}^{n-1} \sigma\big\{Q\in SO(d): Q(x_i) \notin E\cap S_{r_i}(0)\big\}\\
&>1 - (n-1)\Big(1- \frac{n-2}{n-1}\Big) =0.
\end{align*}
Hence there is a set of rotations $Q$ of positive $\sigma$-measure such that $Q(x_i) \in E$ for all $1\leq i\leq n-1$, as required. (Note that this argument remains valid if the $r_i$ are not all distinct.)
\end{proof}

Our main theorem, stating that sets of density greater than $\frac{n-2}{n-1}$ contain all sufficiently large copies of $n$ point patterns, now follows easily.
\bigskip

\begin{proof}[Proof of Theorem \ref{mainth} {assuming Proposition \ref{concenspheres}}]
Taking $\rho' = \frac{n-2}{n-1}$ and $m=n-1$ in Proposition \ref{concenspheres} there is a number $s_0(E,S,D,m)$ such that
for all $s\geq s_0$ there exists $x_0\in E$ such that
$$
\mathcal{L}^{d-1} (E\cap S_{r_i s}(x_0) ) > \Big(\frac{n-2}{n-1}\Big) \mathcal{L}^{d-1} (S_{r_i s})
$$
for all $1\leq i \leq n- 1$, noting that $S\leq r_i \leq D$. Thus, for all $s\geq s_0$, by Lemma \ref{rotation} there is a $Q\in SO(d)$ such that $sQ(P) + x_0 - sQ(x_0) = Q\big(s(P-x_0)\big)+x_0 \subseteq E$.
\end{proof}

\section{Proof of Proposition \ref{concenspheres}}
\label{sec:keyestimates}

{ We first show} that $E$ has mean density not much more than $\rho$ in all sufficiently large balls  but also there exist balls of all large radii where $E$ has mean density close to $\rho$.

\begin{lemma}\label{observation:rho}
Let $E\subseteq \R^d$ be  Lebesgue-measurable with upper Banach density $\rho>0$ and let $\alpha>0$. Then we may find $s_1:=s_1(\alpha,E)$ such that
\be\label{upmd}
\frac{\mathcal{L}^d(E \cap B)}{\mathcal{L}^d(B)}< \rho (1+\alpha),
\ee
for all closed balls $B$ of radii grater than $s_1$.
Furthermore, for all $s>0$, there exists a closed ball $B_s$ such that
\be\label{lomd}
\frac{\mathcal{L}^d(E\cap B_s)}{\mathcal{L}^d(B_s)}>\rho (1-\alpha).
\ee
\end{lemma}

\begin{proof}
Inequality \eqref{upmd} is clear from the definition of $\rho$.

For \eqref{lomd}, given $s>0$, we may find $r>s$ such that $1-(1-\frac{s}{r})^d<\frac{1}{2}\rho\alpha $ and  $x\in \R^d$ satisfying \[\frac{\mathcal{L}^d(E \cap B_r(x))}{\mathcal{L}^d(B_r)}>\rho \Big(1-\frac{\alpha}{2}\Big).\]
Then,
 \[ \int_{B_r(x)} \mathcal{L}^d(B_s(y)\cap E) \, dy\geq \mathcal{L}^d(E\cap B_{r-s}(x)) \mathcal{L}^d(B_s) \]
and
\[\mathcal{L}^d(E \cap B_{r-s}(x))\geq \mathcal{L}^d(E\cap B_r(x))- \Big(1-\Big(1-\frac{s}{r}\Big)^d\Big) \mathcal{L}^d(B_r).\]
Hence,
\begin{align*}
\frac{1}{\mathcal{L}^d(B_r)} \int_{B_r(x)}\frac{\mathcal{L}^d(E\cap B_s(y))}{\mathcal{L}^d(B_s)} \,dy
&\geq \frac{\mathcal{L}^d(E \cap B_{r-s}(x))}{\mathcal{L}^d(B_r)}\\
&\geq \frac{\mathcal{L}^d(E \cap B_r(x))}{\mathcal{L}^d (B_r)}- \Big(1-\Big(1-\frac{s}{r}\Big)^d\Big)\\
&>\rho \Big(1-\frac{\alpha}{2}\Big)-\rho\frac{\alpha}{2}=\rho (1-\alpha).
\end{align*}
So, there exists $y \in B_r(x)$ such that $\frac{\mathcal{L}^d(E \cap B_s(y))}{\mathcal{L}^d(B_s)}>\rho (1-\alpha)$.
\end{proof}

We will need to estimate the $(d-1)$-dimensional measure of the intersection of $(d-1)$-spheres with the set $E$. To facilitate this we approximate such spheres by annuli. Let $A_{r_1,r_2}(x) := B_{r_2}(x) \setminus B_{r_1}(x) $ be the $d$-{\em dimensional annulus} of centre $s$, inner radius $r_1$ and outer radius $r_2$. The intersection of pairs of such annuli is key to our calculations, and for $v\in \R^d$ and $\delta >0$ we define
\be \label{phidef}
\phi_{\delta}^{(d)}(v):=\delta^{-2}\mathcal{L}^d(A_{r,r+\delta}(0)\cap A_{r,r+\delta}(v)).
\ee
We will check that the limit as  $\delta \to 0$ of $\phi_{\delta}^{(d)}(v)$ exists pointwise and in $L^1$ and equals the following function $K_r^{(d)}$ which may be thought of as a potential kernel on $\R^d$.

\begin{definition}
For $r>0$ define $K_r^{(d)}: \R^d \to \R$ by
$$
K_r^{(d)}(v):= \left\{ \begin{array}{lcc}
             \displaystyle{\frac{2r^2  \pi^{(d-1)/2} (r^2 - \frac{\|v\|^2}{4})^{(d-3)/2}}{ \Gamma \left(\frac{d-1}{2} \right)\|v\|} }&   \text{if}  & \|v\|< 2r \text{ and } v\neq 0, \\
\vspace{-0.3cm}            \\ 0 &  \text{if}  & \|v\|> 2r, \\
             \\ +\infty &  \text{if}  & \|v\|= 2r \text{ or } v=0,
             \end{array}
   \right.
$$
where $\Gamma$ is the gamma function.
\end{definition}

Throughout we will write $A_{r}^d$ for  the $(d-1)$-dimensional surface area of a ball $B_r \subseteq \R^d$, given by
\be\label{sufarea}
A_{r}^d : = \frac{d\, r^{d-1} \pi^{d/2}}{  \Gamma (\frac{d}{2}+1)}.
\ee

\begin{lemma}\label{lemlim}
For all $r>0$, $\phi_{\delta}^{(d)}\to K_r^{(d)}$ pointwise and in $L^1(\R^d)$. Furthermore,
\be\label{intk}
\int K_r^{(d)}(v) \ dv =  (A_{r}^d)^2.
\ee
\end{lemma}
\begin{proof}
Pointwise convergence is trivial if $v=0$ or $\|v\| \geq 2r$.

For $0<\|v\|<2r$ first consider the case when $d=2$. The circles $C_r(0)$ and $C_r(v)$ intersect at angle $\theta$ where $\sin\frac{\theta}{2}=\|v\|/2r$. Then for  small $\delta>0$, $A_{r, r+\delta}(0) \cap A_{r, r+\delta}(v)$ is a pair of regions, each close to a rhombus of side $\delta/\sin \theta$ and height $\delta$, so of area  $\delta^2 /\sin \theta$. Hence,
\begin{align*}
\mathcal{L}^2\{A_{r, r+\delta}(0) \cap A_{r, r+\delta}(v)\}&=2\frac{\delta^2}{\sin\theta} + O(\delta^3)=\frac{2\delta^2}{2 \sin\frac{\theta}{2} \cos\frac{\theta}{2}}+ O(\delta^3)\\
&=\frac{\delta^2}{\frac{\|v\|}{2r}\left(1-\frac{\|v\|^2}{(2r)^2} \right)^{1/2}}+ O(\delta^3)=\frac{2\delta^2  r^2}{\|v\| \left( r^2- \frac{\|v\|^2}{4}\right)^{1/2}}+ O(\delta^3)\\
&=\delta^2 K_r^{(2)}(v)+ O(\delta^3),
\end{align*}
noting that $\Gamma(\frac{1}{2}) = \pi^{\frac{1}{2}}$, so pointwise convergence at $v$ when $d=2$ follows noting \eqref{phidef}.

For  $d\geq 3$, we use the half of the estimate when $d=2$, rotating one of the two approximate rhombii.
Let $G_r:=r S^{d-1} \cap (r S^{d-1} + v)$ where $S^{d-1}$ is the unit $(d-1)$-sphere centred at $0$, so $G_r$  is the $(d-2)$-sphere $\tilde{r}S^{d-2}+\frac{1}{2}v$ of radius $\tilde{r}:=\big(r^2 - \frac{\|v\|^2}{4}\big)^{1/2}$ which is contained in the hyperplane $\langle v \rangle^{\perp} +\frac{1}{2}v$.
Then \[\mathcal{L}^{d-2}(G_r)=\mathcal{L}^{d-2}(\tilde{r}S^{d-2})=\frac{2 \pi^{(d-1)/2} \,\tilde{r}^{d-2}}{\Gamma \left(\frac{d-1}{2} \right)}.\]
For $0<\|v\|<2r$,
\begin{align*}
\mathcal{L}^d\{A_{r, r+\delta}(0) \cap A_{r, r+\delta}(v)\}
&=\bigg( \delta^2 \frac{K_r^{(2)}(v)}{2} + O(\delta^3)\bigg)\mathcal{L}^{d-2}(G_r)\\
&=  \delta^2\frac{K_r^{(2)}(v)}{2} \frac{2 \pi^{(d-1)/2}\, \tilde{r}^{d-2}}{\Gamma \left(\frac{d-1}{2} \right)} + O(\delta^3)\\
&= \delta^2 \frac{K_r^{(2)}(v)}{2} \frac{2 \pi^{(d-1)/2} \big(r^2 - \frac{\|v\|^2}{4}\big)^{(d-2)/2}}{\Gamma \left(\frac{d-1}{2} \right)} + O(\delta^3)\\
&=  \delta^2K_r^{(d)}(v) +  O(\delta^3)
\end{align*}
again giving convergence at $v$.

Pointwise convergence of $\phi_{\delta}^{(d)}(v)$ is not uniform, but to establish $L^1$ convergence it is enough to check that $\|\phi_{\delta}^{(d)}\|_{1} \to \|K_r^{(d)}\|_{1}$. Noting that  $\int\mathcal{L}^d(A\cap (B+v)) dv = \mathcal{L}^d(A) \mathcal{L}^d(B)$ for measurable $A,B \subseteq \R^d$, from \eqref{phidef}
\begin{align}
\int \phi_{\delta}^{(d)}(v) \ dv
&= \delta^{-2} \mathcal{L}^d(A_{r,r+\delta}(0))^2\nonumber\\
&=\delta^{-2}(\delta A_r^d +O(\delta^2))^2 \nonumber\\
&=(A_r^d)^2 +O(\delta). \label{phicon}
\end{align}
Using spherical coordinates,
\begin{align}
\int K_r^{(d)}(v) \ dv &=  \frac{2 r^2 \pi^{(d-1)/2}}{\Gamma \left(\frac{d-1}{2} \right)} \left( \int_0^{2r}\Big(r^2-\frac{\rho^2}{4}\Big)^{(d-3)/2}\rho^{d-2} \, d\rho \right)\! A_1^d\nonumber\\
&=\frac{2 r^2 \pi^{(d-1)/2}}{\Gamma \left(\frac{d-1}{2} \right)} 2^{d-2} r^{2d-4}\left( \int_0^{1}(1-t)^{(d-3)/2}\, t^{(d-3)/2} \, dt \right) \! A_1^d \label{int1}\\
&=\frac{2 r^2 \pi^{(d-1)/2}}{\Gamma \left(\frac{d-1}{2} \right)} 2^{d-2} r^{2d-4}\left(\frac{\Gamma (\frac{d-1}{2})^2}{(d-2)!}\right) \! A_1^d   \nonumber\\
&=\frac{d\pi^{d- 1/2} 2^{d-1} r^{2d-2}}{\Gamma (\frac{d}{2}+1)} \frac{\Gamma (\frac{d-1}{2})}{(d-2)!} \label{int3}\\
&= \frac{d^2 r^{2d-2} \pi^d}{ \left( \Gamma (\frac{d}{2}+1)\right)^2}=  (A_{r}^d)^2,  \label{int4}
\end{align}
where we have used the substitution $\rho = 2rt^{1/2}$ to get the integral form of the beta function  $\beta(\frac{d-1}{2},\frac{d-1}{2})$ at \eqref{int1}, followed by  \eqref{sufarea} at \eqref{int3}, and the factorial form of the gamma function at multiples of $\frac{1}{2}$ to get \eqref{int4}.
From \eqref{phicon},
$\|\phi_{\delta}^{(d)}\|_{1} \to \|K_r^{(d)}\|_{1}$ which, together with pointwise convergence, implies that $\phi_{\delta}^{(d)}\to K_r^{(d)}$ in $L^1$.
\end{proof}

For $r>0$ write
$$
g_r(x):=\mathcal{L}^{d-1}(E\cap S_r(x))  \quad (x \in \R^d)
$$
for the measure of intersection of the set $E\subseteq \R^d$ with the sphere $S_r(x)$. The next lemma enables us to find the mean and mean square of $g_r$.

\begin{lemma}\label{lemma:kernelDdim}
Let $E$ be a bounded Lebesgue-measurable subset of $\R^d$ and let $r>0$. Then
\be\label{gr}
\int g_r(x) \, dx= A_{r}^d \, \mathcal{L}^d(E).
\ee
and
\be\label{grsq}
\int g_r(x)^2 \ dx= \int_E \int_E K_r(y-z) \,dy \, dz.
\ee
\end{lemma}

\begin{proof}
Let $g \in L^1(\R^d)$ be continuous and of compact support.  Then,
\begin{align*}
\int \bigg( \int_{A_{r, r+\delta}(x)} g(v)\ dv \bigg) dx &=\iint \chi_{A_{r, r+\delta}(0)}(v-x)g(v) \, dv \, dx\\
&=\iint \chi_{A_{r, r+\delta}(0)}(u) g(x+u) \, du \, dx\\
&=\int \chi_{A_{r, r+\delta}(0)}(u) \, du  \int g(y) \, dy\\
&=\left(\delta A_{r}^d+O(\delta^2) \right)\int g(y) \,dy.
\end{align*}
Dividing by $\delta$ and letting $\delta \to 0$,
\[\int \bigg( \int_{S_r(x)} g(v) \, d\mathcal{L}^{d-1}(v) \bigg) \, dx= A_{r}^d \int g(y) \,dy,\]
where the left-hand side inner integral is with respect to  $(d-1)$-dimensional Lebesgue measure on the sphere.
Identity \eqref{gr} follows on approximating $\chi_E$ by continuous functions $g$.

Now let $g,  h \in L^1(\R^d)$ be bounded and of compact support with $g$ continuous. Then,
\begin{align*}
\int \bigg[ \int g(y) h(y-x) \ dy \bigg]^2 \ dx&= \iiint g(y) h(y-x) g(z) h(z-x) \ dx \ dy \ dz\\
&=\iint g(y) g(z)  \bigg( \int h(u) h(u-y+z) \ du \bigg) \, dy \, dz
\end{align*}
Taking $h(u):=\delta^{-1} \chi_{A_{r, r+\delta}(0)}(u)$,
\begin{align*}
\int \bigg[ \delta^{-1} \int_{A_{r, r+\delta}(x)} g(y) \, dy \bigg]^2 \ dx&= \iint g(y) g(z)\,  \delta^{-2} \mathcal{L}^d\{A_{r, r+\delta}(0) \cap A_{r, r+\delta}(y-z)\}  \,dy \, dz \nonumber\\
&= \iint g(y) g(z) \phi_{\delta}(y-z)\nonumber  \, dy \,dz.
\end{align*}
Letting $\delta \searrow 0$ then, as $g$ is continuous, $\frac{1}{\delta}\int_{A_{r, r+\delta}(x)} g \to \int_{S_r(x)} g$ and
$\phi_{\delta} \to K_r^{(d)} \text{ in } L^1(\R^d)$ by Lemma \ref{lemlim},
\[\int \bigg( \int_{S_r(x)} g(y)\, dy \bigg)^2 \, dx= \iint g(y) g(z) K_r^{(d)}(y-z) \,dy \,dz.\]
Again, approximating $\chi_E$ by continuous $g$ gives \eqref{grsq}.
\end{proof}

The next lemma provides a good upper bound for the right-hand integral of \eqref{grsq} when  $E$ is reasonably uniformly distributed across a region.

\begin{lemma}\label{lemma:intkerboundDdim}
Let $\delta >0, 0<\e_0 <1$ and $0<\xi\leq 1$ be given. Then there exists $\lambda:=\lambda (\e_0, \xi, \delta) \in (0, \e_0)$ such that if $B_s\subseteq \R^d$ is any ball of radius $s>0$ and $E \subseteq B_s$ is any measurable set such that
\be\label{dens}
\frac{\mathcal{L}^d(E \cap B)}{\mathcal{L}^d(B)}< \rho (1+\alpha)
\ee
 for all balls $B \subseteq B_s$ of radius at least $\lambda s$, then for all $\e\in [\xi\e_0,\e_0]$,
$$
\int_E \int_E K_{\e s}^{(d)}(x-y)\ dx \ dy< (A_{\e s}^{d} )^2 (1+\e)^d \big((1+\alpha)^2 \rho^2 +\delta\big) \mathcal{L}^d(B_s).
$$
\end{lemma}

\begin{proof}
By applying a similarity transformation it is enough to prove the lemma in the special case $s=1$ and $B_s=B_1(0)$.
For each $0<\lambda<1$ and $u\in \R^d$ we define $h_{\lambda}(u):=\frac{1}{\mathcal{L}^d (B_{\lambda})} \chi_{B_\lambda(0)}(u)$.

Let $\eta:=\delta (A^d_{\xi\e_0})^2 (1+\xi\e_0)^d \mathcal{L}^d(B_1(0))$.
Choose $\lambda:=\lambda (\e_0, \xi, \delta) \in (0,\e_0)$ sufficiently small so that  for all  $\e\in [\xi\e_0,\e_0]$,
\be\label{happrox}
\int_{B_1(0)} \int_{B_1(0)} \bigg| K_{\e}^{(d)}(x-y) - \iint K_{\e}^{(d)}(z-w) h_{\lambda}(x-z) h_{\lambda}(y-w) \, dz \, dw \bigg| \, dx \, dy < \eta.
\ee
To achieve this, note that the double integral is continuous in $\e$, for example using that $K_{\e'}^{(d)}$ converges to $K_{\e}^{(d)}$ as $\e'\to\e$ pointwise almost everywhere and in $L^1(B_1(0)\times B_1(0))$. We can find a value of $\lambda$ such that \eqref{happrox} is satisfied for each $\e\in [\xi\e_0,\e_0]$ so compactness enables a choice of $\lambda$ valid for all such $\e$.

Let $E \subseteq B_1(0)$ be a measurable set such that \eqref{dens}  holds for all balls $B \subseteq B_1(0)$ of radius at least $\lambda$.
Then, for all $\e\in [\xi\e_0,\e_0]$, restricting the domain of integration in \eqref{happrox}  to  $E\times E \subseteq B_1(0)\times B_1(0)$, we get
\begin{align*}
\int_E \int_E K_{\e}^{(d)}(x-y) \ dx \ dy &< \eta +\int_E \int_E \iint K_{\e}^{(d)}(z-w)h_{\lambda}(x-z) h_{\lambda}(y-w)\, dz \, dw \, dx \, dy\\
&< \eta+ \rho^2 (1+\alpha)^2 \int_{B(0, 1+\lambda)} \int_{B(0, 1+\lambda)} K_{\e}^{(d)}(z-w) \, dz \, dw\\
&\leq \eta+ \rho^2 (1+\alpha)^2 \int_{B(0, 1+\e)} (A_\e^{d})^2  \, dw\\
&\leq (A_\e^{d})^2  (1+\e)^d ((1+\alpha)^2 \rho^2 +\delta) \mathcal{L}^d(B_1(0)),
\end{align*}
where we used \eqref{dens} with the definition of $h_\lambda$, that $\lambda \leq\e$, the integral \eqref{intk}, and the definition of $\eta$.
\end{proof}

The next lemma shows that we can find a ball $B_s$ in which $E$ has mean density close to $\rho$ but also with good estimates for proportions of the surfaces of smaller spheres that intersect $E$. We will then use \eqref{bs2} and \eqref{bs3} to show that $g_{\e s}$ is nearly constant across $B_s$. Recall that
$$ g_r(x):=\mathcal{L}^{d-1}(E\cap S_r(x))  \quad (x \in \R^d).$$

\begin{lemma}\label{lemma:bounds}
Let $E\subseteq \R^d$ be a Lebesgue-measurable set of upper Banach density $\rho>0$, and let  $0< \xi\leq 1$. Then, given $\eta \in (0,1)$, we can find $\e_0>0$ and $s_0>0$ such that for each $s>s_0$ there is a ball $B_s\subseteq \R^d$ of radius $s$ satisfying
\be\label{bs1}
\frac{\mathcal{L}^d(E \cap B_s)}{\mathcal{L}^d(B_s)}> \rho (1-\eta),
\ee
\be\label{bs2}
\frac{\int_{B_s}g_{\e s}(x)\, dx}{\mathcal{L}^d(B_s)} > A^d_{\e s}\, \rho(1-\eta)
\ee
and
\be\label{bs3}
\frac{\int_{B_s}g_{\e s}(x)^2\, dx}{\mathcal{L}^d(B_s)}<  (A^d_{\e s})^2 \rho^2 (1+\eta),
\ee
for all $\e\in [\xi\e_0,\e_0]$.
\end{lemma}

\begin{proof}
Given $\eta \in (0,1)$, we choose positive numbers $\alpha$, $\delta$, $\e_0 \in (0,1)$ small enough to ensure that
\begin{equation}\label{eq:cond1}(1+\e_0)^d ((1+\alpha)^2 \rho^2 +\delta)+  (1-(1-\e_0)^d)< \rho^2 (1+\eta),
\end{equation}
and
\begin{equation}\label{eq:cond2} \rho\alpha + (1-(1-\e_0)^d) < \rho \eta.\end{equation}
Let $\lambda $ be given by Lemma \ref{lemma:intkerboundDdim} for these $\delta, \e_0$ and $\xi$.
Let $s_1:=s_1(\alpha,E)$ from Lemma \ref{observation:rho} and
let $s_0:= s_1/\lambda$.  If $s>s_0$ then $s>s_1$ as $\lambda<1$, and there is a ball $B_s$ of radius $s$ such that
\be\label{eq0}
\frac{\mathcal{L}^d(E\cap B_s)}{\mathcal{L}^d(B_s)}>\rho (1-\alpha).
\ee
By \eqref{eq:cond2}  $\alpha < \eta$ giving \eqref{bs1}.

We now establish \eqref{bs3}.
Let $f_{\e s}(x):=\mathcal{L}^{d-1}((E\cap B_s)\cap S_{\e s}(x))$.
By Lemma \ref{lemma:kernelDdim} applied to $E\cap B_s$,
\be\label{eq1}
\int f_{\e s}(x)^2 \ dx= \int_{E\cap B_s} \int_{E\cap B_s} K_{\e s}(y-z) \,dy \, dz.
\ee
By Lemma \ref{lemma:intkerboundDdim} (which hypotheses are satisfied by definition of $s_1$ and that $ s>s_1/\lambda =s_0$), we get that for all $\e\in [\xi\e_0,\e_0]$,
\be\label{eq2}
\int_{E\cap B_s} \int_{E\cap B_s} K_{\e s}(x-y)\, dx \, dy< (A_{\e s}^{d} )^2 (1+\e )^d ((1+\alpha)^2 \rho^2 +\delta) \mathcal{L}^d(B_s).
\ee

Writing $B_{s-\e s}$ for the ball concentric with $B_s$ and of radius $s-\e s$, then $f_{\e s}(x)=g_{\e s}(x)$ for $x \in B_{s-\e s}$ and $\mathcal{L}^d(B_s \setminus B_{s-\e s})=(1-(1-\e)^d)\mathcal{L}^d(B_s)$. By \eqref{eq1} and \eqref{eq2},
\begin{align*}
\int_{B_s} g_{\e s}(x)^2 \ dx &= \int_{B_{s- \e s}} g_{ \e s}(x)^2 \ dx + \int_{B_s \setminus B_{s- \e s}} g_{\e s}(x)^2 \, dx \\
&\leq \int_{B_s} f_{\e s}(x)^2 \, dx +  (A^d_{\e s})^2  (1-(1-\e)^d)\mathcal{L}^d(B_s)\\
&\leq (A_{\e s}^{d} )^2 \Big[ (1+\e )^d ((1+\alpha)^2 \rho^2 +\delta) +   (1-(1-\e)^d)\Big]\mathcal{L}^d(B_s)\\
&< (A^d_{\e s})^2 \rho^2 (1+\eta)  \mathcal{L}^d(B_s),
\end{align*}
using  \eqref{eq:cond1} since $\e \leq \e_0$.

Finally we apply \eqref{gr} to $E\cap B_{s-\e s}$ to get \eqref{bs2}.
\begin{align*}
\int_{B_s} g_{\e s}(x) \, dx&=\int_{B_s} \mathcal{L}^{d-1}(E\cap S_{\e s}(x)) \, dx\\
&\geq \int_{\R^d}  \mathcal{L}^{d-1}((E\cap B_{s-\e s})\cap S_{\e s}(x)) \, dx\\
&= A^d_{\e s}\,\mathcal{L}^d(E \cap B_{s-\e s})\\
&\geq A^d_{\e s}\,  \big[\mathcal{L}^d(E\cap B_s) - \mathcal{L}^d (B_s \setminus B_{s-\e s})\big]\\
&\geq  A^d_{\e s}\,  \big[ \rho (1-\alpha_0) - \big(1-(1-\e s)^d\big)\big]\mathcal{L}^d(B_s)\\
&>  A^d_{\e s}\, \rho (1-\eta) \mathcal{L}^d(B_s),
\end{align*}
using \eqref{eq0} and \eqref{eq:cond2}.
\end{proof}

The following general lemma bounds the deviation of a function from its mean in terms of its second moment.

\begin{lemma}\label{estimate}
Let $D\subseteq \R^d$ be measurable with $0<\mathcal{L}^d(D) <\infty$, let $g\colon D\to \R_{\geq 0}$ be measurable and not identically $0$, and let $\theta>0$. Then
\begin{align*}
\mathcal{L}^d\bigg\{x\in D: \bigg|g(x) -\frac{1}{\mathcal{L}^d(D)}\int_D \! g(y)dy\bigg|\geq \theta \frac{1}{\mathcal{L}^d(D)}& \int_D \!g(y)dy \bigg\}\\
&\leq\frac{1}{\theta^2}\, \mathcal{L}^d(D) \bigg[\frac{\mathcal{L}^d(D)\int_D g^2}{(\int_D g)^2} -1\bigg].
\end{align*}
\end{lemma}

\begin{proof}
Identically
$$\int_D \left( g(x) -\frac{1}{\mathcal{L}^d(D)} \int_D g (y) \, dy \right)^2 dx
= \int_Dg(x)^2\, dx- \frac{(\int_D g)^2}{\mathcal{L}^d(D)},$$
so by Chebyshev's inequality

\begin{align*}
\mathcal{L}^d\bigg\{x\in D: \bigg|g(x) -\frac{1}{\mathcal{L}^d(D)}\int_D \! g(y)dy\bigg|\geq \theta \frac{1}{\mathcal{L}^d(D)}& \int_D \!g(y)dy \bigg\}\nonumber\\
&\leq \frac{1}{\theta^2} \frac{\mathcal{L}^d(D)^2}{(\int_D g)^2} \bigg[\int_Dg^2 - \frac{(\int_D g)^2}{\mathcal{L}^d(D)}\bigg]\\
&= \frac{1}{\theta^2}\, \mathcal{L}^d(D) \bigg[\frac{\mathcal{L}^d(D)\int_D g^2}{(\int_D g)^2} -1\bigg].
\end{align*}
\end{proof}

Using Lemma \ref{estimate} with the estimates of Lemma \ref{lemma:bounds} we now show that there is a ball $B_s$ such that `most' $(d-1)$-spheres of radius $\e s$ centred inside $B_s$ intersect $E$ in a proportion of the sphere `close to' $\rho$, the
Banach density of $E$, for a suitable range of $\e$.

\begin{proposition}\label{lemma:mesest}
Let $E\subseteq \R^d$ be a Lebesgue-measurable set of upper Banach density $\rho>0$ and let $0<\rho'<\rho$. Let $0<\xi\leq 1$ and $\delta>0$.
Then there exist $s_0>0$ and $\e_0>0$ such that for all $s\geq s_0$ there is a ball $B_s \subseteq\R^d$ such that
\be\label{meslb}
\mathcal{L}^d(E \cap B_s)>\rho'{\mathcal{L}^d(B_s)}
\ee
and
\be\label{mesbad}
\mathcal{L}^d\big\{x\in B_s: g_{\e s}(x)\leq \rho'A^d_{\e s}\big\} <\delta \mathcal{L}^d(B_s)
\ee
for all $\e \in [\xi\e_0,\e_0]$.
\end{proposition}

\begin{proof}
 Let $\rho' = (1-\theta) \rho$ where $0<\theta<1$. Choose $\eta>0$ small enough so that
\be\label{eta}
\frac{4}{\theta^2} \bigg[\frac{1+\eta}{(1-\eta)^2} - 1\bigg] < \delta \quad \mbox{and} \quad \eta < {\textstyle \frac{1}{2}\theta}
\ee
By Lemma \ref{lemma:bounds}, given these $\rho, \xi$ and $\eta$, there exist $\e_0$ and $s_0$ such that for all $s>s_0$ there is a ball $B_s$ satisfying \eqref{meslb} by \eqref{bs1} and \eqref{eta}, and also for all $\e \in [\xi\e_0,\e_0]$,
\begin{align*}
\mathcal{L}^d\big\{x\in B_s: g_{\e s}(x)\leq &\rho(1-\theta) A^d_{\e s}\big\}\\
&\leq \mathcal{L}^d\big\{x\in B_s: g_{\e s}(x)\leq \rho(1-  {\textstyle \frac{1}{2}\theta})(1-\eta) A^d_{\e s}\big\}\\
&\leq \mathcal{L}^d\bigg\{x\in B_s: g_{\e s}(x)\leq (1-  {\textstyle \frac{1}{2}\theta})\frac{\int_{B_s}g_{\e s}(x)\, dx}{\mathcal{L}^d(B_s)} \bigg\}\\
&= \mathcal{L}^d\bigg\{x\in B_s: \frac{\int_{B_s}g_{\e s}(x)\, dx}{\mathcal{L}^d(B_s)}- g_{\e s}(x) \geq {\textstyle \frac{1}{2}\theta}\,\frac{\int_{B_s}g_{\e s}(x)\, dx}{\mathcal{L}^d(B_s)} \bigg\}\\
&\leq \frac{4}{\theta^2}\, \mathcal{L}^d(B_s) \bigg[\frac{\mathcal{L}^d(B_s)\int_{B_s} g_{\e s}(x)^2}{(\int_{B_s} g_{\e s}(x))^2} -1\bigg]\\
&\leq \frac{4}{\theta^2}\, \mathcal{L}^d(B_s) \bigg[\frac{1+\eta}{(1-\eta)^2} - 1\bigg]\\
&<\delta \mathcal{L}^d(B_s),
\end{align*}
where we have used \eqref{eta}, \eqref{bs2}, Lemma \ref{estimate}, \eqref{bs2} and \eqref{bs3}, and \eqref{eta}.
\end{proof}

Now we are able to prove Proposition \ref{concenspheres}
\begin{proof}[{Proof of Proposition \ref{concenspheres}}]
Given $E$, choose  $0<\delta < \rho'/m$ and set $\xi = S/D$. Let $s_0$ and $\e_0$ be given by Proposition \ref{lemma:mesest} for these values. Thus for all $s\geq s_0$ there is a ball $B_s$ such that \eqref{meslb} and \eqref{mesbad} hold for all $\e \in [\e_0S/D, \e_0]$.
By scaling by a factor $\e_0/D$ it is enough to assume that $r_i\in [\e_0S/D, \e_0]$ for all $i$. Then
\begin{align*}
\mathcal{L}^d\big\{x\in E\cap B_s: & \ g_{r_i s}(x)\geq \rho'A^d_{r_i s} \mbox{ for all } 1\leq i \leq m\big\} \\
&\geq \mathcal{L}^d(E\cap B_s) - \sum_{i=1}^{m}\mathcal{L}^d\big\{x\in E\cap B_s: g_{r_i s}(x)\leq \rho'A^d_{r_i s}\big\}\\
&\geq \rho'\mathcal{L}^d( B_s) -m\delta\, \mathcal{L}^d(B_s) >0.
\end{align*}
Thus for all $s\geq s_0$ we may choose $x\in E\cap B_s$  such that \eqref{bigint} is satisfied for all $i$.
\end{proof}

\section{Proof of Theorem \ref{thm:lowerbound}: lower bound for $\rho_{\textup{min}}$}
\label{sec:lowerbound}

In this section we will prove a lower bound claimed in Theorem \ref{thm:lowerbound}, that is, construct a set of density at least that stated in \eqref{eq:lowerbound} that does not contain all sufficiently large $n$-term arithmetic progressions.

\begin{proof}[Proof of Theorem \ref{thm:lowerbound}]
Note that the claim is void unless $10\log n/n^{1/5}<1$. Thus, we assume that $n$ is large enough so that this holds and set
\[ \varepsilon := \frac{10\log n}{n^{1/5}} \in(0,1). \]
In particular, we will have $n>10^5$ throughout the proof.

The set $E$ will come from Bourgain's construction in \cite{B86}, that is, it will be a `thin' version of the set from Example \ref{ex:Bourgain}. We define
\begin{align*}
E & := \bigcup_{m=0}^{\infty} \Big\{ x\in\mathbb{R}^d : m - \frac{1-\varepsilon}{2} < \|x\|^2 < m + \frac{1-\varepsilon}{2} \Big\} \\
& = \Big\{ x\in\mathbb{R}^d : \dist( \|x\|^2, \mathbb{Z} ) < \frac{1-\varepsilon}{2} \Big\}.
\end{align*}
It is easy to see that $E$ has (the most usual type of) density equal to $1-\varepsilon$.

Take some $r>0$ and suppose that there exists an isometry $\R\to E\subseteq\R^d$ mapping $kr\mapsto x_k$ for $k=0,1,\ldots,n-1$, where $x_0,x_1,\ldots,x_{n-1}$ are some points in the set $E$. Let $a_k := \|x_k\|^2$.
The parallelogram law gives
\[ 2\big(\|x_{k}\|^2 + \|x_{k+2}\|^2\big) = \|2x_{k+1}\|^2 + \|x_{k+2}-x_{k}\|^2, \]
i.e.,
\[ a_{k+2} - 2a_{k+1} + a_{k} = 2r^2. \]
Solving this recurrence relation easily gives
\begin{equation}\label{eq:quadraticsequence}
a_k = r^2 k^2 + A k + B; \quad k=0,1,2,\ldots
\end{equation}
for some constants $A$ and $B$. Note that we are constrained to indices $k\leq n-1$ only, but the above formula extends and defines an infinite sequence $(a_k)_{k=0}^{\infty}$.
For now we only assume that $r^2$ is an irrational number; later we will refine this choice.

We will consider the sequence
\begin{equation}\label{eq:sequencemod1}
\mathbf{a} = \big(a_k \textup{ mod } 1\big)_{k=0}^{\infty}
\end{equation}
on the one-dimensional torus $\mathbb{T}=\mathbb{R}/\mathbb{Z}\equiv[0,1)$, so that we can apply quantitative results on uniform distribution of sequences. These results belong to the realm of \emph{discrepancy theory} \cite{Bil14,DT97,KN74}, also known as \emph{single-scale equidistribution theory} \cite[\S 1.1.2]{Tao12}.
The main idea is the following:
\begin{itemize}
\item On the one hand, by the construction of $E$, the first $n$ terms of the sequence $\mathbf{a}$ completely avoid the interval $[(1-\varepsilon)/2,(1+\varepsilon)/2]\subseteq\mathbb{T}$ of length $\varepsilon$.
\item On the other hand, for sufficiently large $n$ the first $n$ terms of the sequence $\mathbf{a}$ should be `sufficiently uniformly distributed' over $\mathbb{T}$.
\end{itemize}
These two claims will lead to a contradiction.
For the second claim we could use some result on quantitative equidistribution of polynomial sequences on $\mathbb{T}$, such as Exercise 1.1.21 from Tao's book \cite{Tao12}.
However, since our sequence \eqref{eq:quadraticsequence} is very special (i.e., it is only quadratic), {since we need uniformity in the coefficients $A,B$} and since we want to be entirely quantitative (i.e., with a precise exponent $1/5$ and an explicit constant, such as $10$), we {will} redo some of the theory from scratch.

The \emph{discrepancy} of the first $n$ terms of the sequence \eqref{eq:sequencemod1} is the number
\[ D_n(\mathbf{a}) := \sup_{[\alpha,\beta)\subseteq[0,1)} \Big| \frac{\mathop{\textup{card}}\big\{ k\in\{0,1,\ldots,n-1\} :  a_k \textup{ mod }1 \in [\alpha,\beta) \big\}}{n} - (\beta-
\alpha) \Big|, \]
which quantifies how uniformly the $a_k$ are distributed over $\mathbb{T}$.
Once we can guarantee
\begin{equation}\label{eq:ineqcontra}
D_n(\mathbf{a}) < \varepsilon,
\end{equation}
we will arrive at a contradiction by taking $[\alpha,\beta]=[(1-\varepsilon)/2,(1+\varepsilon)/2]$.
The famous Erd\H{o}s--Tur\'{a}n inequality (see \cite[Chapter~2, Theorem~2.5]{KN74}) gives an explicit upper bound for the discrepancy in terms of exponential sums:
\begin{equation}\label{eq:Dnestproof1}
D_n(\mathbf{a}) \leq \frac{6}{M+1} + \frac{4}{\pi}\sum_{m=1}^{M} \frac{1}{m} \Big| \frac{1}{n} \sum_{k=0}^{n-1} e^{2\pi i m a_k} \Big|
\end{equation}
for any positive integer $M$ (to be chosen later).

Next, we use an explicit version of van der Corput's trick for exponential sums (see \cite[Lemma~1.1.6]{Tao12}) to get for yet another positive integer $H$ (to be chosen later):
\begin{equation}\label{eq:Dnestproof2}
\Big| \frac{1}{n} \sum_{k=0}^{n-1} e^{2\pi i m a_k} - \frac{1}{n} \sum_{k=0}^{n-1} \frac{1}{H} \sum_{h=1}^{H} e^{2\pi i m a_{k+h}} \Big| \leq \frac{2H}{n}.
\end{equation}
We now estimate the above double sum. We use the Cauchy-Schwarz inequality for the sum in $k$, expand out the square, take into account the explicit formula \eqref{eq:quadraticsequence}, and sum up a few finite geometric sequences:
\begin{align}
\Big| \frac{1}{n} \sum_{k=0}^{n-1} \frac{1}{H} \sum_{h=1}^{H} e^{2\pi i m a_{k+h}} \Big|^2
& \leq \frac{1}{n} \sum_{k=0}^{n-1} \Big| \frac{1}{H} \sum_{h=1}^{H} e^{2\pi i m a_{k+h}} \Big|^2 \nonumber \\
& = \frac{1}{H} + \frac{2}{H^2 n} \mathop{\textup{Re}} \sum_{\substack{0\leq k\leq n-1\\ 1\leq h<h'\leq H}} e^{2\pi i m (a_{k+h'}-a_{k+h})} \nonumber \\
& \left[\,\text{substitute } j=k+h,\ l=h'-h\,\right] \nonumber \\
& = \frac{1}{H} + \frac{2}{H^2 n} \mathop{\textup{Re}} \sum_{l=1}^{H-1} e^{2\pi i l m(lr^2+A)} \sum_{h=1}^{H-l} \sum_{j=h}^{n+h-1} e^{4\pi i j l mr^2} \nonumber \\
& \leq \frac{1}{H} + \frac{4}{H n} \sum_{l=1}^{H-1} \frac{1}{|1-e^{4\pi i l m r^2}|}. \label{eq:Dnestproof3}
\end{align}

The final ingredient comes from the theory of Diophantine approximations \cite{Bug12}. Let us choose a \emph{badly approximable} $z\in[0,1)$, which means that
\begin{equation}\label{eq:badapprox1}
\Big|z-\frac{p}{q}\Big| \geq \frac{c}{q^2}
\end{equation}
for some $c=c(z)>0$ and all $p,q\in\mathbb{Z}$, $q\neq 0$. One such choice is the golden ratio
\begin{equation}\label{eq:badapprox2}
z=\frac{-1+\sqrt{5}}{2},
\end{equation}
in which case we can take
\begin{equation}\label{eq:badapprox3}
c=\frac{1}{3}.
\end{equation}
This can be seen in an entirely elementary way, by using Vi\`{e}te's formulae and writing
\begin{align*}
\frac{1}{q^2} \leq \frac{|p^2+pq-q^2|}{q^2}
& = \Big|\frac{p}{q}-\frac{-1+\sqrt{5}}{2}\Big| \Big|\frac{p}{q}-\frac{-1-\sqrt{5}}{2}\Big| \\
& \leq \Big|\frac{p}{q}-z\Big| \bigg( \Big|\frac{p}{q}-z\Big| + \sqrt{5} \bigg).
\end{align*}
A consequence of \eqref{eq:badapprox1}--\eqref{eq:badapprox3} is
\begin{equation}\label{eq:badapprox4}
\dist(qz,\mathbb{Z}) \geq \frac{1}{3q}
\end{equation}
for every positive integer $q$.
It is interesting to remark that uncountably many choices of $z$ would work out here, provided that we were willing to lower the constant \eqref{eq:badapprox3} to $2^{-15}$, see \cite[Theorem~7.8]{Bug12}.

Now let $r>0$ be any number such that $r^2 - z \in \mathbb{Z}$, where $z$ was given in \eqref{eq:badapprox2}. The set of such numbers is unbounded. From \eqref{eq:badapprox4} we get
\[ |1-e^{4\pi i l m r^2}| = |1-e^{4\pi i l m z}| \geq 4\,\dist(2 l m z, \mathbb{Z}) \geq \frac{2}{3 l m} \]
for positive integers $l$ and $m$, so
\begin{equation}\label{eq:Dnestproof4}
\frac{4}{H n} \sum_{l=1}^{H-1} \frac{1}{|1-e^{4\pi i l m r^2}|} \leq \frac{3H m}{n}.
\end{equation}
Combining \eqref{eq:Dnestproof1}, \eqref{eq:Dnestproof2}, \eqref{eq:Dnestproof3}, and \eqref{eq:Dnestproof4} we end up with
\[ D_n(\mathbf{a}) \leq \frac{6}{M+1} + \frac{4}{\pi}(1+\log M) \Big( \frac{2H}{n} + \frac{1}{H^{1/2}} \Big) + \frac{8\sqrt{3}}{\pi}\Big(\frac{H M}{n}\Big)^{1/2}, \]
so choosing
\[ H = \lfloor (1/25) n^{2/5} \rfloor, \quad M = \lfloor 4 n^{1/5} \rfloor \]
we obtain
\[ D_n(\mathbf{a}) < \frac{10\log n}{n^{1/5}}. \]
This is precisely \eqref{eq:ineqcontra} and it leads to the desired contradiction.
\end{proof}

\end{document}